\documentclass{amsart}

\usepackage{amsthm,amssymb}
\usepackage{amsmath,amscd}
\usepackage{amssymb}

\usepackage{pinlabel}
\usepackage{lscape}
\usepackage{latexsym}
\usepackage[utf8]{inputenc}

\newtheorem{Thm}{Theorem}[section]
\newtheorem{Cor}[Thm]{Corollary}
\newtheorem{Conj}[Thm]{Conjecture}
\newtheorem{Prop}[Thm]{Proposition}
\newtheorem{Lem}[Thm]{Lemma}

\newtheorem*{thma}{Theorem A}
\newtheorem*{thmb}{Theorem B}

\theoremstyle{definition}
\newtheorem{Def}[Thm]{Definition}
\newtheorem{Ex}[Thm]{Example}

\theoremstyle{remark}

\numberwithin{equation}{section}

\newcommand{\Aut}{\operatorname{Aut}}

\newcommand{\Hom}{\operatorname{Hom}}
\newcommand{\Mor}{\operatorname{Mor}}
\newcommand{\Id}{\operatorname{Id}}

\newcommand{\Syl}{\operatorname{Syl}}
\newcommand{\Iso}{\operatorname{Iso}}
\newcommand{\Out}{\operatorname{Out}}

\newcommand{\Ob}{\operatorname{Ob}}

\renewcommand{\Gamma}{\varGamma}
\renewcommand{\epsilon}{\varepsilon}
\renewcommand{\bar}{\overline}

\renewcommand{\leq}{\leqslant}
\renewcommand{\geq}{\geqslant}

\newcommand{\D}{\mathcal{D} }

\newcommand{\F}{\mathcal{F}}
\renewcommand{\L}{\mathcal{L}}

\newcommand{\U}{\mathcal{U}}

\newcommand{\E}{\mathcal{E}}
\newcommand{\C}{\mathcal{C}}

\begin{document}


\title{Centralizers of Subsystems of Fusion Systems}
 

\author{J. Semeraro}
\address{Heilbronn Institute for Mathematical Research, Department of Mathematics, University of Bristol, U.K.}
\email{js13525@bristol.ac.uk}
\maketitle 


%


\begin{abstract}
When $(S,\F,\L)$ is a $p$-local finite group and $(T,\E,\L_0)$ is weakly normal in $(S,\F,\L)$ we show that a definition of $C_S(\E)$ given by Aschbacher has a simple interpretation from which one can deduce existence and strong closure very easily. We also appeal to a result of Gross to give a new proof that there is a unique fusion system $C_\F(\E)$ on $C_S(\E)$.
\end{abstract}
\section{Introduction}
Let $p$ be prime. By a \textit{normal pair} $((S,\F),(T,\E))$ of saturated fusion systems we will mean that $\F$ is a saturated fusion system on a finite $p$-group $S$ and $\E$ is a normal subsystem of $\F$ on $T \leq S$.   The centralizer fusion system $C_\F(\E)$ was first considered by Aschbacher in \cite[Chapter 6]{A2} and is a fusion system of index a power of $p$ in $C_\F(T)$ on $C_S(\E)$, the largest subgroup $X$ of $C_S(T)$ with the property that $\E \subseteq C_\F(X).$ While this definition is natural and easy-to-state, the proof of uniqueness is rather technical and relies on several of the tools developed in  \cite{A1} and earlier chapters of \cite{A2}. The purpose of this note is twofold. Firstly, we show that $C_S(\E)$ may be regarded as the centralizer of a \textit{linking system} of $\E$, where existence, uniqueness and strong closure follow almost trivially. Secondly, we apply Aschbacher's characterization of $C_S(\E)$ as an intersection of local subsystem centralizers in order to give a new proof of the existence of $C_\F(\E)$.

\subsection{Main results and structure of the paper.}
Recall that a $p$-local finite group is a triple $(S,\F,\L)$ where $S$ is a finite $p$-group, $\F$ is a saturated fusion system on $S$ and $\L$ is a linking system associated to $\F$ (see \cite[Definition III.4.4]{AKO}). When $((S,\F),(T,\E))$ is a weakly normal pair of fusion systems, and $\L_0 \subseteq \L$ are the associated linking systems, $\L_0$ is weakly normal in $\L$ if a certain pair of conditions hold, mirroring the weak normality conditions for fusion systems \cite[Definition III.4.12]{AKO}. This definition naturally gives rise to the concept of a weakly normal pair  
$((S,\F,\L), (T,\E,\L_0))$ of $p$-local finite groups. We prove the following result in § \ref{s:centlink}, which may be regarded as a generalization of \cite[Lemma 1.14]{AOV}:

\begin{thma}\label{weaklythm}
Let $((S,\F,\L), (T,\E,\L_0))$ be a weakly normal pair of $p$-local finite groups. Then $C_S(\E)$ fits into an exact sequence 
$$\begin{CD}
1 @>>> C_S(\E) @>>> \Aut_\L(T) @>>> \Aut_{typ}^I(\L_0) @>>> \hspace{20mm}   \\
\end{CD}$$
$$\begin{CD}
\hspace{80mm} \Out_{typ}^\L(\L_0) @>>> 1 \\
\end{CD}.$$  In particular, $C_S(\E)$ is a well-defined strongly $\F$-closed subgroup of $S$.
\end{thma}

The definitions of $\Aut_\L(T)$, $\Aut_{typ}^I(\L_0)$ and $\Out_{typ}^\L(\L_0)$ are all provided or referenced in § \ref{s:centlink}.  
A remarkable feature of Theorem A is that, in contrast to the corresponding result of \cite[Chapter 6]{A2} for fusion systems, the proof is very short and follows quickly from the definitions.

Our next main result, presented in § \ref{s:centfus}, is a new proof of the existence of the fusion system centralizer $C_\F(\E)$. Recall that that the \textit{hyperfocal} subgroup, $\mathfrak{hyp}(\F)$ of a saturated fusion system $\F$ has the property that whenever $\mathfrak{hyp}(\F) \leq R \leq S$, there is a unique saturated subsystem $\F_R$ of $\F$ on $R$ contained in $\F$  at index a power of $p$ (see \cite[Theorem 4.3]{BCGLO2}). We use this fact to prove the following:

\begin{thmb}
Let $((S,\F), (T,\E))$ be a normal pair of saturated fusion systems. Then $$\mathfrak{hyp}(C_\F(T)) \leq C_S(\E) \leq C_S(T).$$ In particular, there is a unique saturated subsystem $C_\F(\E)$ on $C_S(\E)$ contained in $C_\F(T)$ at index a power of $p$.
\end{thmb}

For Theorem B, we appeal to the characterization of $C_S(\E)$ as an intersection of subsystem centralizers and combine this with an observation of Gross in \cite{FlG}.

In § \ref{s:zstar}, we discuss some topics for future work, focussing mainly on a conjectured generalization of Glauberman's celebrated $Z^*$-theorem.

\section{Background}\label{prelim}
\subsection{Group Theory}
Standard notation and terminology from finite group theory which can be found in any good text on the subject (e.g. \cite{AshFG}) will be adopted. We collect some particularly important notions for future reference. Recall that for a group $G$, $O^p(G)$ (respectively $O^{p'}(G)$) is the smallest normal subgroup $H$ of $G$ with $G/H$ a $p$- (respectively $p'$-) group. Similarly, $O_p(G)$ (respectively $O_{p'}(G)$) is the largest normal $p$- (respectively $p'$-) subgroup of $G$. We say that $G$ has a \textit{normal $p$-complement} if $O^{p}(G)=O_{p'}(G)$. A subgroup $H$ of $G$ is \textit{$G$-centric} if $C_G(H) \leq H$. When $O_{p'}(G)=1$, $G$ is \textit{$p$-constrained} if there exists a normal $G$-centric $p$-subgroup of $G$.

\subsection{Fusion systems}
We assume the reader is familiar with the basic definitions and terminology associated with fusion systems which can be found in either of the two recent texts on this subject \cite{AKO,CR}. In particular, we assume a working knowledge of saturation and Alperin's fusion theorem, together with definitions and results concerning `important' subsystems such as $\F$-normalizers and $\F$-centralizers of subgroups, and various examples of normal subsystems.

\subsection{Linking systems}
For an introduction to the theory of abstract linking systems, including a definition and some of its basic consequences, we refer the reader to \cite[III.4.1]{AKO}. An alternative viewpoint (which avoids the categorical definition) is presented in \cite[§ 2]{Ch}. A proof that these two viewpoints are equivalent can be found in \cite[Appendix A]{Ch}. 
\newline
\newline

\section{Centralizers of Linking Systems}\label{s:centlink}
In this section, we shall adopt the notation and terminology of \cite[§ III.4]{AKO}. Let $\F$ be a saturated fusion system on a finite $p$-group $S$. For the definition of a \textit{linking system} associated to $\F$ see \cite[Definition 4.1]{AKO}. If $((S,\F),(T,\E))$ is a weakly normal pair of fusion systems and $\L_0$ and $\L$ are linking systems associated to $\E$ and $\F$ respectively, then \cite[Definition 4.12]{AKO} explains what it means for $\L_0$ to be \textit{weakly normal} in $\L$. We begin by describing how $\L$ `acts' on $\L_0$  by conjugation. 

Recall the definitions of \textit{isotypical}, $\Aut_{typ}^I(\L)$ and $\Out_{typ}(\L)$ given in \cite[§ III.4.3]{AKO}. We construct a ``conjugation'' map
\begin{equation}\label{cgam}
\begin{CD}
\mbox{Aut}_\L(T) @>\gamma \mapsto c_\gamma>> \mbox{Aut}_{typ}^I(\L_0) 
\end{CD}
\end{equation} as follows: For each $\gamma \in \Aut_\L(T)$, let $c_\gamma=c_\gamma|_{\L_0}$ be the functor from $\L_0$ to itself defined by setting $(P)c_\gamma:=P\gamma:=P\pi(\gamma)$ for $P \in \Ob(\L_0)$ and $(\psi)c_\gamma:=(\gamma|_{P,P\gamma})^{-1} \circ \psi \circ (\gamma|_{Q,Q\gamma}) \in \Hom_{\L_0}(P\gamma,Q\gamma)$ for each $\psi \in \Hom_{\L_0}(P,Q)$.

Note that both of the functors $\pi$ and $\delta$ act on the left (as is conventional), while we let $c_\gamma$ act on the right so that (\ref{cgam}) defines a group homomorphism. 
 Note also that the above construction makes implicit use of the fact that $\L_0$ is weakly normal in $\L$. We make the following observation:

\begin{Lem}
Let $((S,\F,\L),(T,\E,\L_0))$ be a weakly normal pair of $p$-local finite groups. Then $c_\gamma|_{\L_0} \in \Aut_{typ}^I(\L_0)$ for each $\gamma \in \Aut_\L(T)$.
\end{Lem} 

\begin{proof}
By \cite[Proposition 4.3 (a)]{AKO} applied to $\L$, $\iota_P^Q \circ \gamma|_{Q,Q\gamma}=\gamma|_{P,P\gamma} \circ \iota_{P\gamma}^Q$ for each $P,Q \in$ Ob$(\L_0)$ so that  $c_\gamma$ clearly sends inclusions to inclusions. To see that $c_\gamma$ is isotypical, note that for each $P \in$ Ob$(\L_0)$ and $g \in P$,  $$\delta_P(g) \circ \gamma|_{P,P\gamma}=\gamma|_{P,P\gamma} \circ \delta_{P\gamma}(g\pi(\gamma))$$ by Axiom (C) for $\L$ (\cite[Definition 4.1]{AKO}). Thus $(\delta_P(g))c_\gamma=\delta_{P\gamma}(g\pi(\gamma))$, as needed. 
\end{proof}

Next, we generalize the set $\Out_{typ}(\L_0)$ to define an object dependent on a weakly normal pair of $p$-local finite groups:

\begin{Def}
Let $((S,\F,\L),(T,\E,\L_0))$ be a weakly normal pair of $p$-local finite groups. Two isotypical equivalences $\alpha,\beta$ are \textit{$\L$-naturally isomorphic} if for each $R \in \Ob(\L_0)$, there exists $\eta_R \in \Hom_\L(\alpha(R),\beta(R))$ such that for all $\varphi \in \Hom_{\L_0}(R,R'),$ the diagram $$
\begin{CD}
\alpha(R) @>\eta_R>> \beta(R)\\
@VV\alpha(\varphi)V @VV\beta(\varphi)V\\
\alpha(R') @>\eta_{R'}>> \beta(R')
\end{CD}$$ commutes. 
Let $\Out_{typ}^\L(\L_0)$ denote the set of equivalence classes of $\Aut_{typ}(\L_0)$ under $\L$-natural isomorphism.
\end{Def}

Notice that $\L$-natural isomorphism is a coarser equivalence relation than that of natural isomorphism. Thus one has a natural map $$\begin{CD}
\Out_{typ}(\L_0) @>>> \Out^\L_{typ}(\L_0).\\
\end{CD}
$$ We will need one more definition before we can state and prove Theorem A.

\begin{Def}\label{csedef}
Let $((S,\F),(T,\E))$ be a weakly normal pair of fusion systems. Then
$$C_S(\E):=\{g \in S  \mid \E \subseteq C_\F(\langle g \rangle)\}.$$
\end{Def}

We now prove Theorem A, which may be regarded as a generalization of \cite[Lemma 1.14(a)]{AOV}.

\begin{Thm}\label{weaklythm}
Let $((S,\F,\L),(T,\E,\L_0))$ be a weakly normal pair of $p$-local finite groups. The sequence 
$$\begin{CD}
1 @>>> C_S(\E) @>\delta_S>> \Aut_\L(T) @>\gamma \mapsto c_\gamma>> \Aut_{typ}^I(\L_0) @>>> \Out_{typ}^\L(\L_0)\\
\end{CD}$$ is exact. In particular, $C_S(\E)$ is a strongly $\F$-closed subgroup of $S$.
\end{Thm}

\begin{proof}
By \cite[Lemma III.4.9]{AKO}, $\Aut_{typ}^I(\L_0)$ is a group and so  $C_S(\E)$ is a group provided the sequence is exact. First we show that the natural homomorphism from $\Aut_{typ}^{I}(\L_0)$ to $\Out_{typ}^\L(\L_0)$ is onto. This amounts to showing that each $\alpha \in \Aut_{typ}(\L_0)$ is $\L$-naturally isomorphic to an isotypical equivalence which sends inclusions to inclusions. But \cite[Lemma III.4.9]{AKO} implies that any such $\alpha$ is $\L_0$-naturally isomorphic to an isotypical equivalence, establishing the claim.

Now suppose that $\alpha \in  \Aut_{typ}^I(\L_0)$ is in the kernel of the natural map from  $\Aut_{typ}^{I}(\L_0)$ to $\Out_{typ}^\L(\L_0)$. Then $\alpha$ is $\L$-naturally isomorphic to the identity functor and so for each $P \in \Ob(\L_0)$, there are $\eta_P \in  \Iso_\L(P,\alpha(P))$, and for each $\psi \in  \Hom_{\L_0}(P,Q)$, $$\eta_P \circ \alpha(\psi)=\psi \circ \eta_Q.$$ Since $\alpha$ is isotypical, $\alpha(\iota_P^T)=\iota_{\alpha(P)}^T$ so that $\eta_P=\eta_T|_{P,\alpha(P)}$, and this shows that $\alpha$ is given by conjugation by $\eta_T \in   \Aut_\L(T)$. Conversely, if $\gamma \in \Aut_\L(T)$ then $c_\gamma$ is $\L$-naturally isomorphic to the identity functor $\Id_{\L_0}$ via the set of maps $\{\gamma|_{P, P\pi(\gamma)}\}_{P \in Ob(\L_0)}.$

Next we show that for each $\gamma \in  \Aut_\L(T)$, $c_\gamma= \Id_{\L_0}$ if and only if $\gamma \in \delta_S(C_S(\E))$. If $c_\gamma= \Id_{\L_0}$ then since $\gamma^{-1} \circ \delta_T(g) \circ \gamma=\delta_T(g)$ for all $g \in S$, we must have $\pi(\gamma)=  \Id_T$ by Axiom (C) for $\L$ (\cite[Definition 4.1]{AKO}) and the fact that $\delta_T$ is injective. Hence by Axiom (A2) for $\L$, there is some $a \in C_S(T)$ such that $\gamma=\delta_T(a)$. Now, for each $P,Q \in  \Ob(\L_0)$ and $\psi \in  \Mor_{\L_0}(P,Q)$, $\psi \circ \delta_Q(a)=\delta_P(a) \circ \psi$ so that by \cite[Proposition 4.3(b)]{AKO}, there is some $\bar{\psi} \in  \Mor_\L(\langle P,a \rangle, \langle Q,a \rangle)$ such that $\bar{\psi}|_{P,Q}=\psi$ (note that $a \in C_S(T) \leq C_S(P)$). By Axiom (C) for $\L$ again and the injectivity of $\delta$, $\pi(\bar{\psi})(a)=a$. Hence each morphism in $\E$ extends to one which fixes $a$ and $a \in C_S(\E)$. Conversely if $\gamma=\delta_S(a)$ for some $a \in C_S(\E)$, then $\pi(\gamma)=c_a$ is the identity on $\Ob(\L_0)$ and each $\psi \in$ Mor$(\L_0)$ extends to some $\bar{\psi}$ with $(a)\pi(\bar{\psi})=a$. By Axiom (C) for $\L$, $\bar{\psi}$ commutes with $\delta_S(a)$ and we have $c_\gamma(\psi)=\psi$. Hence $c_\gamma=$ Id$_{\L_0},$ as required.

It remains to prove that $C_S(\E)$ is strongly $\F$-closed. Let $a \in C_S(\E)$ and choose any subgroup $P$ and morphism $\varphi \in \Hom_\F(P,S)$ with $a \in P$. Let $\psi \in \Hom_\L(P,S)$ be such that $\pi(\psi)=\varphi$ and note that by Axiom (C) for $\L$, we have $\psi^{-1} \circ \delta_P(a) \circ \psi=\delta_S(a\varphi).$ This implies that $c_{\delta_S(a\varphi)}=c_{\psi}^{-1} \circ \mbox{ Id}_{\L_0} \circ c_{\psi}=\mbox{ Id}_{\L_0}$ (since $c_{\delta_P(a)}= \Id_{\L_0}$) and $a\varphi \in C_S(\E)$ as needed.
\end{proof}

We end this section with an example which shows that Theorem A does not hold for weakly normal pairs of fusion systems:

\begin{Ex}\label{countex}
Let $H:=H_1 \times H_2 \times H_3$ with $H_i \cong A_4$ for $1 \leq i \leq 3$ be a direct product of three copies of the alternating group on 4 letters. For $1 \leq i \leq 3$ let $S_i \in$ Syl$_2(H_i)$ so that $S:=S_1 \times S_2 \times S_3 \in$ Syl$_2(H)$ and let $X_i = \langle x_i \rangle$ be a group of order 3 which acts on $S_i$ in such a way that $H_i=S_i \rtimes X_i$ for each $i$. Set $X:=\langle x_1x_2, x_1x_3 \rangle \cong C_3 \times C_3$, $G:=SX$ and define $$\F:=\F_S(G),\mbox{ } \F_1:=\F_{S_1S_2}(\langle S_1,S_2, x_1x_2 \rangle) \mbox{ and } \F_2:=\F_{S_1S_3}(\langle S_1,S_3, x_1x_3 \rangle).$$
Clearly $\F_i$ is normal in $\F$ for $i=1,2$ and hence the fusion system $\E:=\F_1 \cap \F_2$ on $S_1$, whose morphisms consist of those in both $\F_1$ and $\F_2$, is an $\F$-invariant subsystem on $S_1$ (intersections of $\F$-invariant subsystems are $\F$-invariant). Furthermore $\E=\F_{S_1}(H_1)$ so $\E$ is saturated and hence weakly normal. However, $\E$ is not normal in $\F$. To see this, let $\varphi \in$ Aut$_\E(S_1)$ be the map induced by conjugation by $x_1$ so that $\varphi$ extends to maps $\varphi_1=c_{x_1x_2}$ and $\varphi_2=c_{x_1x_3} \in$ Aut$_\F(S)$. Observe that $\varphi_i$ does not act trivially on $C_S(S_1)/Z(S_1) = S/S_1 = S_2 \times S_3$ for $i=1,2$, proving the claim.

Now, for $i=2,3$, $\E \subseteq C_\F(S_i)=\F_{4-i} \times \F_{S_i}(S_i)$, but $\E$ is not contained in $C_\F(S_2  \times S_3)=\F_S(S)$. Hence there is no unique maximal subgroup of $S$ which centralizes $\E$, and hence no weakly normal pair of linking systems associated to $(\F,\E)$ by Theorem A.
\end{Ex}

\section{Centralizers of Fusion Systems}\label{s:centfus}

In order to introduce Aschbacher's local characterization of $C_S(\E)$, we first introduce local subsystems for fusion systems. In this section, $((S,\F),(T,\E))$ is always assumed to be a normal pair of saturated fusion systems.

\begin{Def}
Let $U$ be a fully $\F$-normalized $T$-centric subgroup of $T$, and let $X \leq C_S(T)$. Define $U_X:=UXC_S(UX),$ and let
 $\D(U,X):=N_\F^K(U_X)$, where $$K:=\{\varphi \in \Aut_\F(U_X) \mid \varphi|_{UX} \in \Aut_\F(UX) \}.$$
\end{Def}

The reader will readily check that $\D(U,1)$ is exactly the fusion system $\D(U)$ described in \cite[Definition 8.37]{CR}. The next result, which is a summary of \cite[Lemma 6.3]{A2}, may be regarded as a generalization of \cite[Lemma 8.38]{CR}, at least in the case where $U$ is $T$-centric.

\begin{Lem}\label{eudu}
Let $U$ be a fully $\F$-normalized $T$-centric subgroup of $T$, and let $X \leq C_S(T)$. If $UX$ is fully $N_\F(U)$-normalized then the following hold:

\begin{itemize}
\item[(a)] $((N_S(UX),\D(U,X)),(N_T(U),N_\E(U)))$ is a normal pair of constrained fusion systems realized by a normal pair $(G(U,X),H(U,X))$ of finite groups.
\item[(b)]  $H(U):=H(U,1)$ is a subgroup of $G(U,X)$ and $$C_S(H(U)) = C_S(H(U,X)).$$
\end{itemize}
\end{Lem}

\begin{proof}
(a) is \cite[Lemma 6.3.3]{A2}, while (b) is \cite[Lemma 6.3.5]{A2}.
\end{proof}

It turns out that one can formulate $C_S(\E)$ as an intersection of centralizers of the $H(U)$, as the next result describes:

\begin{Thm}\label{csegroup}
Let $\U$ be the set of all fully $\F$-normalized, $T$-centric subgroups of $S$. Then $$C_S(\E)= \bigcap_{\varphi \in \Aut_\F(TC_S(T))} \left(  \bigcap_{U \in \U} C_S(H(U)) \right) \varphi.$$ Furthermore, $C_S(\E)$ is a strongly $\F$-closed subgroup of $S$.
\end{Thm}

\begin{proof}
This is shown in \cite[Chapter 6]{A2}.
\end{proof}

\subsection{Existence of $C_\F(\E)$}
A definition of $C_\F(\E)$, due to Aschbacher, has already appeared in Chapter 6 of \cite{A2}, and the results in this section rely on some of the machinery developed there. Our approach does not however use the theory of normal maps (see \cite[§ 7]{A1}). Instead we rely on a group theoretic result, which is a corollary to the following theorem of Gross in \cite{FlG}:

\begin{Thm}\label{autcentsylow}
Let $G$ be a $p$-constrained finite group with Sylow $p$-subgroup $S$. Assume that $O_{p'}(G)=1$ and write $$C:=C_{\Aut(G)}(S)=\{\varphi \in \Aut(G) \mid \varphi|_S=\Id_S\}.$$ Then $C$ has a normal $p$-complement.
\end{Thm}

\begin{proof}
This follows from (1) and (2) in \cite[§ 5]{FlG}: if $G$ is a $p$-constrained, minimal counterexample to the lemma, (1) and (2) imply that $O_p(G)=1$, where the minimality of $G$ is applied with the respect to the groups $Z/Z(G)$ and $C_G(M)/M$ with $M=Z(O_p(G))$. This supplies the required contradiction. (Note that the assumption $p > 2$ is not used anywhere in either of these two steps.)
\end{proof}

\begin{Cor}\label{corauto}
Let $T \leq S$ be finite $p$-groups and $(G,H)$ be a normal pair of $p$-constrained finite groups with Sylow $p$-subgroup $(S,T)$. If $O_{p'}(G)=1$ then $$O^p(C_G(T)) \cap S \leq C_G(H).$$
\end{Cor}

\begin{proof}
Let $C_{\Aut_G(H)}(T)$ be the image of $C_G(T)$ under the natural map
$$\begin{CD}
\Phi: C_G(T) @>>> \Aut_G(H). \end{CD}$$ Theorem \ref{autcentsylow} implies that $C_{\Aut_G(H)}(T)$ has a normal $p$-complement. In particular, $O^p(C_{\Aut_G(H)}(T))$ is a $p'$-group. Since $$(O^p(C_G(T))\cap S)C_G(H)/C_G(H) \leq O^p(C_G(T))C_G(H)/C_G(H) \leq O^p(C_G(T)/C_G(H)),$$ and this latter group is a $p'$-group, we must have $O^p(C_G(T))\cap S \leq C_G(H)$, as needed.
\end{proof}

Recall the definition of the hyperfocal subgroup of a saturated fusion system:

\begin{Def}\label{defhyp} For any saturated fusion system $\F$ over a finite $p$-group $S$, the hyperfocal subgroup $\mathfrak{hyp}(\F)$ of $\F$ is given by:
$$\mathfrak{hyp}(\F):=\langle g^{-1} \cdot g\alpha \mid g \in P \leq S, \alpha \in O^p(\Aut_\F(P))\rangle$$
\end{Def}

As an immediate consequence of the hyperfocal subgroup theorem (\cite[Theorem 4.3]{BCGLO2}) and \cite[Theorem 1]{A1} we obtain: $$\mathfrak{hyp}(C_\F(T)) \leq C_S(\E)$$ whenever $((S,\F),(T,\E))$ is a pair of \textit{constrained} fusion systems. In particular, there is a saturated fusion system $C_\F(\E)$ on $C_S(\E)$ contained in $C_\F(T)$ by \cite[Theorem I.7.4]{AKO} in this case. Our goal is to prove this in the case where $\E$ and $\F$ are not constrained. We will achieve this by combining Corollary \ref{corauto} with Theorem \ref{csegroup}. First, we need a definition:

\begin{Def}
For each $X \leq C_S(T)$, a chain of subgroups $$\C(U):=U=U_0 < U_1 < \cdots < U_n=T$$ is \textit{strongly $(\F,X)$-normalized} if $U_{i+1}=N_T(U_i)$, $U_i$ is fully $\F$-normalized and $U_iX$ is fully $N_\F(U_i)$-normalized for each $0 \leq i < n$.
\end{Def}

Observe that strongly $(\F,1)$-normalized is exactly what Aschbacher calls strongly $\F$-normalized in \cite[Chapter 1]{A2}. We thus seek a generalization of \cite[1.1.1]{A2}. In other words, under suitable conditions, we would like to assert that each $U \leq T$ is $\F$-conjugate to a subgroup $W$ which affords a strongly $(\F,X)$-normalized chain $\C(W)$. For this, we need another lemma of Aschbacher which says (among other things) that local subsystems behave well under taking normalizers. Write $\U$ for the set of all fully $\F$-normalized $T$-centric subgroups of $T$.

\begin{Lem}\label{aschex}
Let $X \leq C_S(T)$ and $U \in \U$ be such that $UX$ is fully $N_\F(U)$-normalized, and define $Q:=N_T(U)$. There exists $\alpha \in \Hom_\F(N_S(Q),S)$ with the following properties:
\begin{itemize}
\item[(a)] $U\alpha,Q\alpha \in \U$, $(XU)\alpha$ is fully $N_\F(U\alpha)$-normalized and  $(XQ)\alpha$ is fully $N_\F(Q\alpha)$-normalized;
\item[(b)] $\alpha$ extends to a map $\check{\alpha}: G(U,X) \rightarrow G(U\alpha,X\alpha)$ with the property that $$N_{G(Q\alpha,X\alpha)}(U\alpha)=N_{G(U,X)}(Q)\check{\alpha} \mbox{ and } N_{H(Q\alpha,X\alpha)}(U\alpha)=N_{H(U,X)}(Q)\check{\alpha};$$
\end{itemize}
\end{Lem}

\begin{proof}
This follows from parts $(6)$, $(7)$ and $(8)$ of \cite[Lemma 6.6.3]{A2}, where we observe that the condition $X \subseteq C_S(\E)$ is never used.
\end{proof}

Following \cite[Notation 6.8]{A2}, for each $U \in \U$ and $X \leq C_S(T)$ for which $UX$ is fully $N_\F(U)$-normalized, define: $$C(U,X):=C_{G(U,X)}(N_T(U)) \mbox{ and } K(U,X):=O^p(C(U,X)).$$
We next present a slight simplification of \cite[6.10]{A2}, which shows that we can pass $p$-local information between elements of $\U$. 

\begin{Lem}\label{strongn}
Let $X \leq C_S(T)$ and $U \in \U$ be such that $$\C(U):=U=U_0 < U_1 < \cdots < U_n=T$$ is a strongly $(\F,X)$-normalized chain. For each $0 \leq j < n$, there exists an isomorphism 

$$\begin{CD}
\theta_j: N_{G(U_j,X)}(U_{j+1}) @>>> N_{G(U_{j+1},X)}(U_j), \\
\end{CD}$$
with \begin{equation}\label{kuj}
C(U_j,X)\theta_j \cdots \theta_n=C(T,X) \mbox{ and } K(U_j,X)\theta_j \cdots \theta_n = K(T,X).
\end{equation}

\end{Lem}

\begin{proof}
For each $0 \leq j \leq n-1,$ the isomorphism $\theta_j$ is the isomorphism $\check{\alpha}$ of Lemma \ref{aschex} (b) applied with $U=U_j$ and $Q=U_{j+1}$. It remains to prove that (\ref{kuj}) holds. Writing $$G_j:=G(U_j,X),  H_j:=H(U_j,X), C_j:=C(U_j,X) \mbox{ and  }  K_j:=K(U_j,X), $$ we have  $$[C_j\theta_j,U_{j+1}] \leq [C_j\theta_j,H_{j+1}] \leq C_{H_{j+1}}(U_{j+1}) = Z(U_{j+1}),$$ so that  $C_j\theta_j \leq C_{j+1}$. Since, $K_j\theta_j \leq C_{G_j}(U_{j+1})\theta_j$  also $K_j\theta_j \leq K_{j+1}$. Conversely $C_{j+1}\theta_j^{-1}$ and $K_{j+1}\theta_j^{-1}$ both centralize $U_{j+1}$ which implies that $C_j\theta_j=C_{j+1}$ and $K_j\theta_j=K_{j+1}$. This completes the proof.
\end{proof}

\begin{Cor}\label{coraschex}
Let $X \leq C_S(T)$ and $U \in \U$ be such that $XU$ is fully $N_\F(U)$-normalized. Then there exists $\alpha \in \Hom_\F(U,S)$ such that $\C(U\alpha)$ is strongly $(\F,X\alpha)$-normalized. 
\end{Cor}

\begin{proof}
 Assuming the lemma is false, let $U$ be a counterexample with $n:=|T:U|$ as small as possible and write $Q:=N_T(U)$. Since $T$ is fully $\F$-normalized and $XT$ is fully $N_\F(T)$-normalized by assumption when $n=1$, we see that $n > 1$. 

By Lemma \ref{aschex} (a), there is some $\alpha \in \Hom_\F(N_S(Q),S)$ such that $U\alpha,Q\alpha \in \U$ and $(XU)\alpha$ and $(XQ)\alpha$ are fully $N_\F(U\alpha)$- and $N_\F(Q\alpha)$-normalized respectively. Furthermore, by (\ref{kuj}) in Lemma \ref{strongn} above, $$X\alpha \leq C_S(T)\alpha \leq C_S(Q)\alpha \leq C_S(Q\alpha)=C_S(T).$$

By induction $\C(Q\alpha)$ is strongly $(\F,X\alpha)$-normalized, and hence so is $\C(U\alpha)$, contradicting the minimal choice of $n$.
\end{proof}

From Definition \ref{defhyp}, we have the following result concerning the hyperfocal subgroup of $C_\F(T)$:

\begin{Prop}\label{hypcont}
For each $U \in \U$, $\mathfrak{hyp}(C_\F(T)) \leq 
C_{N_S(U)}(H(U))$.
\end{Prop}

\begin{proof}
Let $X \leq C_S(T)$ and  $\beta \in O^p(\Aut_{C_\F(T)}(X))$. It suffices to show that $[X,\beta] \leq C_{N_S(U)}(H(U))$. By Corollary \ref{coraschex} there exists $\alpha \in \Hom_\F(N_S(Q),S)$ such that $\C(U\alpha)$ is strongly $(\F,X\alpha)$-normalized with $Q:=N_T(U)$. If $[X\alpha,\beta^\alpha] \leq C_{N_S(U\alpha)}(H(U\alpha))$ then $$[X,\beta]=[X\alpha,\beta^\alpha]\check{\alpha}^{-1} \leq C_{N_S(U\alpha)}(H(U\alpha))\check{\alpha}^{-1}= C_{N_S(U)}(H(U)).$$ Thus, on replacing $(U,X)$ by $(U\alpha,X\alpha)$ if necessary, we may assume that $\C(U)$ is strongly $(\F,X)$-normalized.
Suppose we have shown that \begin{equation}\label{opcft} O^p(\Aut_{C_\F(T)}(X))=\Aut_{K(T,X)}(X). \end{equation} Then, since $X \leq N_S(U_j)$ for each $0 \leq j \leq n$, Lemma \ref{strongn} implies that $\Aut_{K(T,X)}(X)=\Aut_{K(U,X)}(X)$ and $$[X,\beta] \leq \mathfrak{hyp}(C_{\D(U,X)}(N_T(U))).$$ Now, the hyperfocal subgroup theorem (\cite[Theorem 4.3]{BCGLO2}) implies that this latter group is equal to $O^p(C_{G(U,X)}(N_T(U))) \cap N_S(UX)$, which is contained in  $C_S(H(U,X))$ by Corollary \ref{corauto}. Since  $$C_S(H(U,X))=C_S(H(U,1))=C_S(H(U))$$ by Lemma \ref{eudu} (b), the result follows.

It thus remains to prove (\ref{opcft}). Clearly $\Aut_{K(T,X)}(X) \leq O^p(\Aut_{C_\F(T)}(X))$ by definition. Conversely, each $p'$-automorphism $\alpha \in \Aut_{C_\F(T)}(X)$ extends to a morphism $\bar{\alpha} \in \Aut_\F(C_S(XT)XT)$ which fixes $T$ (note that $XT$ is fully $N_\F(T)$-normalized), which we may assume is also of $p'$-order. Thus $\bar{\alpha}=c_g$ for some $g \in K(T,X)$ and $$O^p(\Aut_{C_\F(T)}(X)) \leq \Aut_{K(T,X)}(X),$$ as required.
\end{proof}

\begin{Cor}\label{c: hyps}
$\mathfrak{hyp}(C_\F(T)) \leq C_S(\E).$
\end{Cor}

\begin{proof}
This follows immediately from Theorem \ref{csegroup} and Proposition \ref{hypcont}, provided $\mathfrak{hyp}(C_\F(T))$ is left invariant by elements of $\Aut_\F(TC_S(T))$. But such elements clearly normalize  $O^p(\Aut_{C_\F(T)}(P))$ for each $P \leq C_S(T)$, as needed.
\end{proof}

\begin{Thm}
Let $T \leq S$ be finite $p$-groups and $((S,\F),(T,\E))$ be a normal pair of saturated fusion systems. There is a unique saturated subsystem $C_\F(\E)$ on $C_S(\E)$ contained in $C_\F(T)$ at index a power of $p$.
\end{Thm}

\begin{proof}
This follows from Corollary \ref{c: hyps} and \cite[Theorem I.7.4]{AKO}.
\end{proof}

\section{Open questions}\label{s:zstar}
The $Z^*$-theorem of Glauberman (\cite[Theorem 1]{Gl}), which has since been shown to hold for all primes (\cite[Theorem 1]{Art}),  has the following elementary statement for fusion systems:

\begin{Thm}\label{zstar}
Let $G$ be a finite group with $S \in \Syl_p(G)$ If $O_{p'}(G)=1$ then $Z(\F_S(G))=Z(G)$. 
\end{Thm}

Currently (when $p$ is odd) the proof relies on the classification of finite simple groups, although in recent years there has been some interest in finding a classification-free proof (see \cite{Wa}). One could hope to generalize Theorem \ref{zstar} as follows:

\begin{Conj}\label{cseconj}
Let $((S,\F),(T,\E))$ be a normal pair of fusion systems realized by a normal pair $(G,H)$ of finite groups. If $O_{p'}(H)=1$ then $C_S(\E)=C_G(H) \cap S$.
\end{Conj}

So far, all attempts to construct a counterexample to this conjecture have failed. For example, via elementary calculations, it is true when $(G,H)=(S_n,A_n)$, $n \geq 6$ and $p=2$. Pushing this slightly further, we can ask whether the following is true:

\begin{Conj}\label{cseconj2}
Let $((S,\F),(T,\E))$ be a normal pair of fusion systems realized by a normal pair $(G,H)$ of finite groups. If $O_{p'}(H)=1$ then $$C_\F(\E)=\F_{C_S(H)}(C_G(H))$$
\end{Conj}

To prove Conjecture \ref{cseconj2}, it would be enough to show that $\F_{C_S(H)}(C_G(H))$ is a fusion system of $p$-power index in $C_\F(T)$.

\bibliography{refs}        

\begin{thebibliography}{10}

\bibitem{AOV}
Kasper Andersen, Bob Oliver, and Joana Ventura.
\newblock Reduced, tame and exotic fusion systems.
\newblock {\em Proc. London Math. Soc.}, 105 (1):87--152, 2012.

\bibitem{Art}
O.D. Artemovich.
\newblock Isolated elements of prime order in finite groups.
\newblock {\em Ukrain. Mat. Zh}, 40:397--400, 1988.
\newblock Translation in Ukrainian Math. J. 40 (1988), 343--345.

\bibitem{AshFG}
Michael Aschbacher.
\newblock {\em Finite Group Theory}.
\newblock Cambridge University Press, 1986.

\bibitem{A1}
Michael Aschbacher.
\newblock Normal subsystems of fusion systems.
\newblock {\em Proc. London Math. Soc.}, 97:239--271, 2008.

\bibitem{A2}
Michael Aschbacher.
\newblock The generalized fitting subsystem of a fusion system.
\newblock {\em Mem. Amer. Math. Soc.}, 209, 2011.

\bibitem{AKO}
Michael Aschbacher, Radha Kessar, and Bob Oliver.
\newblock {\em Fusion Systems in Algebra and Topology}.
\newblock Cambridge University Press, 2011.

\bibitem{BCGLO2}
C.~Broto, N.~Castellana, J.~Grodal, R.~Levi, and B.~Oliver.
\newblock Extensions of $p$-local finite groups.
\newblock {\em Trans. Amer. Math. Soc}, 359:3791--3858, 2007.

\bibitem{Ch}
A.~Chermak.
\newblock Fusion systems and localities.
\newblock {\em Acta Math.}, 211:47--139, 2013.

\bibitem{CR}
David~A. Craven.
\newblock {\em The Theory of Fusion Systems}.
\newblock Cambridge studies in advanced mathematics; 131, 2011.

\bibitem{Gl}
G.~Glauberman.
\newblock Central elements in core-free groups.
\newblock {\em J. Algebra.}, 4:403--420, 1966.

\bibitem{FlG}
Fletcher Gross.
\newblock Automorphisms which centralize a sylow $p$-subgroup.
\newblock {\em J. Alg.}, 77:202--233, 1982.

\bibitem{Wa}
R.~Waldecker.
\newblock {\em Isolated involutions in finite groups}, volume 226.
\newblock Mem. Amer. Math. Soc, 2013.

\end{thebibliography}
\bibliographystyle{plain}
\end{document}